\documentclass[12pt,twoside,reqno,psamsfonts]{amsart}
\usepackage{amssymb}
\makeindex
\def\R{{\mathbb R}}

\def\Z{{\mathbb Z}}
\def\N{{\mathbb N}}

\def\T{{\mathbb T}}
\def\squareforqed{\hbox{\rlap{$\sqcap$}$\sqcup$}}
\def\qed{\ifmmode\squareforqed\else{\unskip\nobreak\hfil
\penalty50\hskip1em\null\nobreak\hfil\squareforqed
\parfillskip=0pt\finalhyphendemerits=0\endgraf}\fi}
\newtheorem{thm}{Theorem}[section]
\newtheorem{cor}[thm]{Corollary}
\newtheorem{lem}[thm]{Lemma}
\newtheorem{prop}[thm]{Proposition}
\newtheorem{rem}[thm]{Remark}

\begin{document}
\title{Unitary equivalence of lowest dimensional reproducing formulae of  type $\mathcal{E}_2 \subset Sp(2,\R)$}
\author{R. Boyer}
\address{\textnormal{Department of Mathematics,
Drexel University, 3250 Chestnut Street, Philadelphia, PA 19104, USA} \newline
\textnormal{e-mail: rboyer@math.drexel.edu}}
\author{K. Nowak}
\address{\textnormal{Department of Computer Science,
Drexel University, 3141 Chestnut Street, Philadelphia, PA 19104, USA} \newline
\textnormal{e-mail: kn33@drexel.edu}}
\author{M. Pap}
\address{\textnormal{Faculty of Sciences, University of P\'ecs, 7634 P\'ecs, Ifj\'us\'ag \'ut 6, HUNGARY} \newline
\textnormal{e-mail: papm@gamma.ttk.pte.hu}}

\begin{abstract}
All two-dimensional reproducing formulae, i.e. of $L^2(\R^2)$, resulting out of restrictions of the projective metaplectic representation to connected Lie subgroups of $Sp(2,\R)$  and of type $\mathcal{E}_2$, were listed and classified up to conjugation within $Sp(2,\R)$  in \cite{DM&Co2}, \cite{DM&Co3}. A full classification, up to conjugation within $\R^2 \rtimes Sp(1,\R)$, of one-dimensional reproducing formulae, i.e. of $L^2(\R)$, resulting out of restrictions of the extended projective metaplectic representation to connected Lie subgroups of $\R^2 \rtimes Sp(1,\R)$ was obtained in \cite{DMNo1}, \cite{DMNo2}. In dimension one, there are no reproducing formulae with one-dimensional parametrizations, yet in dimension two, there are reproducing formulae with two-dimensional parametrizations.  Two-dimensional reproducing subgroups of $Sp(2,\R)$  of type $\mathcal{E}_2$ are a novelty. They exhibit completely new phase space phenomena. We show, that they are all unitarily equivalent via natural choices of coordinate systems, and we derive the consequences of this equivalence.
\end{abstract}

\maketitle
\markboth{Unitary equivalence}{R. Boyer, K. Nowak, M. Pap}

\section{Introduction}
Let $(\mathfrak{P},\nu)$ be a measure space, and $\left\{ \phi_\mathfrak{p}\right\}_{\mathfrak{p}\in \mathfrak{P}}$ a measurable field with values in a Hilbert space $\mathcal{H}$ (see e.g. Section 5.3 of \cite{AAG}). We say that $\left\{\phi_\mathfrak{p}\right\}_{\mathfrak{p}\in \mathfrak{P}}$ is a {\it  reproducing system} in $\mathcal{H}$, with the {\it parameter measure} $\nu$, if for every $f\in \mathcal{H}$
\begin{equation}
f=\int_\mathfrak{P}\langle f,\phi_\mathfrak{p} \rangle \phi_\mathfrak{p}\, d\nu(\mathfrak{p}),
\label{reproducing_system_definition}
\end{equation}
where the convergence in (\ref{reproducing_system_definition}) is understood in the weak sense. Via polarization formula
 (\ref{reproducing_system_definition}) is equivalent to 
\begin{equation}
||f||^2=\int_\mathfrak{P}\left| \langle f,\phi_\mathfrak{p} \rangle \right|^2 \, d\nu(\mathfrak{p}),
\label{reproducing_system_polarization}
\end{equation}
valid for all $f\in \mathcal{H}$. Form (\ref{reproducing_system_polarization}) of (\ref{reproducing_system_definition}) is more convenient than  (\ref{reproducing_system_definition}) in formal arguments and we will use it frequently. In case $\nu$ is the counting measure the system $\left\{\phi_\mathfrak{p}\right\}_{\mathfrak{p}\in \mathfrak{P}}$ is called a {\it Parseval frame}. Formulae of the form (\ref{reproducing_system_definition}) are called {\it reproducing formulae}.

The group $Sp(d,\R)$ consist of $2d\times 2d$ invertible matrices, with real coefficients, preserving the symplectic form. The {\it extended projective metaplectic representation} $\mu_e$ of $\R^{2d}\rtimes Sp(d,\R)$ assigns to an affine transformation $g\in \R^{2d}\rtimes Sp(d,\R)$ of the phase space $\R^{2d}=\left\{ (x,\xi)|\,x,\xi \in \R^d \right\}$ the corresponding unitary operator $\mu_e(g)$ acting on $L^2(\R^d)$. The definition of  $\mu_e$, we provide next, is based on the Wigner distribution. There are many alternative models for defining the extended projective metaplectic representation. The choice of the model depends on specific targets. In the current context we want to stress the phase space geometry  phenomena captured by the Wigner distribution. Operator $\mu_e(g)$ translates the affine action of $g$ performed on the Wigner distribution $W_\phi$, $\phi \in L^2(\R^d)$, to the level of $\phi$, i.e. to any $\phi \in L^2(\R^d)$ it assigns $\mu_e(g)\,\phi$ via the formula
\begin{equation}
W_{\mu_e(g)\,\phi}(x,\xi)=W_{\phi}\left(g^{-1}\cdot (x,\xi)\right),
\label{affine_invariance_of_Wigner}
\end{equation}
where $W_\phi (x,\xi)=\int_{\R^d}e^{-2\pi i \langle \xi,y \rangle}
\phi(x+y/2)\overline{\phi(x-y/2)}\,dy$. Function $\mu_e(g)\,\phi$ of 
formula (\ref{affine_invariance_of_Wigner}) is defined up to a phase factor, since the Wigner distribution identifies square integrable functions up to phase factors. As the outcome of (\ref{affine_invariance_of_Wigner}), we obtain the {\it  extended projective metaplectic representation} $\mu_e$. The name extended comes from the fact that we add phase space translations represented as $\R^{2d}$ to the linear action of $Sp(d,\R)$. The name metaplectic is usually used for the (non-projective, i.e. exact, as far as the phase factors are concerned) representation of the double cover of $Sp(d,\R)$, satisfying (\ref{affine_invariance_of_Wigner}). For a comprehensive treatment of the metaplectic representation, from the point of view of analysis in phase space, we refer the reader to books by Folland \cite{Fol1}, Gr\"ochenig \cite{Gro}, De Gosson \cite{DeGo}, and the survey article by De Mari-De Vito \cite{DM&Co4}.

The classical interpretation of the Wigner distribution identifies it as the best possible surrogate of the non-existent joint probability distribution of position and momentum. It is a well established expectation of the mathematical physics community that reproducing formulae should be in one to one correspondence with phase space coverings obtained via the Wigner distribution. It is therefore of primary importance to identify and investigate all reproducing formulae for $L^2(\R^d)$ constructed out of restrictions of the extended metaplectic representation to connected Lie subgroups of $\R^{2d}\rtimes Sp(d,\R)$. A subgroup of $\R^{2d}\rtimes Sp(d,\R)$  is called {\it reproducing}, if it is possible to construct a reproducing formula out of its action on $L^2(\R^d)$, just
by properly choosing the generating function. All one-dimensional, i.e. with $d=1$, reproducing formulae of this type were classified up to a conjugation by an affine transformation of the time-frequency plane in \cite{DMNo1}, \cite{DMNo2}. As a particular consequence, the classification demonstrated, that in one dimension none of the one-dimensional connected Lie subgroups  of $\R^{2}\rtimes Sp(1,\R)$ is reproducing. The situation is different in two dimensions, i.e. for $d=2$. It is possible to construct reproducing formulae for $L^2(\R^2)$ out of restrictions of the extended projective metaplectic representation to two-dimensional connected Lie subgroups of $\R^4\rtimes Sp(2,\R)$.
All reproducing formulae, constructed out of restrictions of the projective metaplectic representation of $Sp(2,\R)$ to connected Lie subgroups of type $\mathcal{E}_2$, were classified up to a conjugation by a linear transformation of the phase space $\R^4$ in \cite{DM&Co2} and \cite{DM&Co3}. It came as a surprise that in two dimensions also two-dimensional reproducing subgroups are possible. The other possible dimensions of reproducing subgroups are three and four. All one-dimensional reproducing formulae classified in \cite{DMNo1}, \cite{DMNo2} may be interpreted geometrically as corresponding to coverings of the time-frequency plane constructed via the action of the reproducing subgroup applied to a compact set. The same phase space geometric interpretation is valid for all other standard reproducing formulae, but not for the two-dimensional reproducing Lie subgroups of $Sp(2,\R)$ of type $\mathcal{E}_2$. In these special cases the set providing the phase space covering via the action of the reproducing subgroup must be non-compact. We refer the reader to \cite{NoPa} for detailed descriptions of phase space coverings corresponding to the two-dimensional lifts of Shannon wavelets, the systems adaptable via unitary maps to all of the representations treated in the current paper.  

Let $\mathcal{Q}$ be the standard maximal parabolic subgroup of $Sp(d,\R)$ consisting of matrices of the form
\begin{equation}
\left[
\begin{array}{cc}
h & 0 \\
\sigma h & { }^th^{-1}
\end{array}
\right],
\label{standard_maximal_parabolic}
\end{equation}
where $h\in GL(d,\R)$, $\sigma \in \text{Sym}(d,\R)$. Let us recall that $GL(d,\R)$ consists of all $d\times d$ 
invertible matrices, and  $\text{Sym}(d,\R)$ of all $d\times d$ symmetric  matrices. In both cases the coefficients are real. Any 
$g\in \mathcal{Q}$ may be factored out as
\begin{equation}
g=
\left[
\begin{array}{cc}
1 & 0 \\
\sigma & 1
\end{array}
\right]
\left[
\begin{array}{cc}
m & 0 \\
0 & { }^tm^{-1}
\end{array}
\right]
\left[
\begin{array}{cc}
a & 0 \\
0 & a^{-1}
\end{array}
\right],
\label{factorization_of_maximal_parabolic}
\end{equation}
where $a\in \R$, $a\ne 0$, and $m\in SL(d,\R)$, i.e. $m\in GL(d,\R)$, $\det m =1$. Formulas 
(\ref{standard_maximal_parabolic}), (\ref{factorization_of_maximal_parabolic}) show that 
$$
\mathcal{Q}=\text{Sym}(d,\R)\rtimes GL(d,\R),
$$
with the group law
$$
\left( \sigma, h \right)\cdot \left( \sigma ', h' \right)=
\left( \sigma+{ }^th^{-1}\sigma ' h^{-1}, hh'\right).
$$
A subgroup of  $\mathcal{Q}$  is called type $\mathcal{E}_d$ if it is of
of the form $\Sigma \rtimes H$ , where $0\ne \Sigma \subset 
\text{Sym}(d,\R)$ is a vector space, and $1\ne H\subset GL(d,\R)$ a connected Lie subgroup. For a group of type $\mathcal{E}_d$ represented as $\Sigma \rtimes H$, we have a very explicit form of the {\it projective metaplectic representation}, namely
\begin{equation}
\mu_e(\sigma,h)f(x)=\left| \det h \right|^{-\frac 12}e^{-2\pi i \Phi(x)\sigma}f\left( h^{-1}x \right),
\label{sigma_d_symbol}
\end{equation}
where for $x\in \R^d$, functional $\Phi(x)\in \Sigma^*$ is defined as $\Phi(x)\sigma=-\frac 12 \sigma x\cdot x$. Function $\Phi:\R^d \rightarrow \Sigma ^*$ is called the {\it symbol} associated to $\Sigma$. 

Table 1 presents a complete list of non-conjugate, two-dimensional reproducing groups of type $\mathcal{E}_2$, obtained in \cite{DM&Co2}, \cite{DM&Co3}. Conjugacy is defined via inner automorphisms of $Sp(2,\R)$.
\begin{equation*}
\begin{matrix}
\text{Subgroup Type}  & \Sigma  &  H  &  \Phi  \\ 
\text{} & u\in \R & t\in \R & \text{} \\
\text{ } & \text{ } & \text{ } & \text{ } \\
\text{I},\alpha \in [-1,0) 
& 
\left[
\begin{array}{cc}
u & 0 \\ 
0 & 0
\end{array}
\right]
&
\left[
\begin{array}{cc}
e^{\alpha t} & 0 \\
0 & e^{(\alpha +1)t}
\end{array}
\right]
& -\frac 12 x_1^2 \\
\text{ } & \text{ } & \text{ } & \text{ } \\
\text{II} 
&
\left[
\begin{array}{cc}
u & 0 \\
0 & 0
\end{array}
\right]
&  e^t
\left[
\begin{array}{cc}
1 & 0 \\
t & 1
\end{array}
\right]
& -\frac 12 x_1^2 \\
\text{ } & \text{ } & \text{ } & \text{ } \\
\text{III},\alpha \in [0,\infty) 
&
\left[
\begin{array}{cc}
u & 0 \\
0 & u
\end{array}
\right]
&  e^t
\left[
\begin{array}{cc}
\cos \alpha t & \sin \alpha t \\
-\sin \alpha t & \cos \alpha t
\end{array}
\right]
& -\frac 12 ( x_1^2 + x_2^2) \\
\text{ } & \text{ } & \text{ } & \text{ } \\
\text{IV},\alpha \in [0,\infty) 
&
\left[
\begin{array}{cc}
u & 0 \\
0 & -u
\end{array}
\right]
&  e^t
\left[
\begin{array}{cc}
\cosh \alpha t & \sinh \alpha t \\
\sinh \alpha t & \cosh \alpha t
\end{array}
\right]
& -\frac 12 ( x_1^2 - x_2^2) 
\end{matrix}
\end{equation*}
$$
\text{Table 1}
$$

In Table 1 rows describe the choices of representatives of non-conjugate conjugacy classes of subgroups. Columns identify parameters of direct products and the projective metaplectic representation. 

The main result of the current paper is summarized in Table 2. For each of the cases listed in Table 1 we identify a coordinate system providing unitary equivalence with case $\text{I },\alpha = -1$.
\begin{equation*}
\begin{matrix}
\text{Subgroup Type}  & U  & \text{ New Coordinates}  
&  \text{Original }\mathcal{H}   &  \text{Resulting }\mathcal{H}\\ 
\text{ } & \text{ } & \text{ } & \text{ }  & \text{ }\\
\text{I },\alpha \in [-1,0) 
& 
y_1^{\frac{\alpha +1}{2\alpha}}f_c(y_1,y_2)
&
\begin{cases}
y_1 =x_1 \\
y_2 =x_1^{-\frac{\alpha +1}{\alpha}}x_2
\end{cases}
& L^2(\R_+ \times \R)  &  L^2(\R_+ \times \R) \\
\text{ } & \text{ } & \text{ } & \text{ } & \text{ } \\
\text{II}
&
y_1^{\frac 12}f_c(y_1,y_2)
& 
\begin{cases}
y_1 =x_1 \\
y_2 ={\frac{x_2 -x_1\log x_1}{x_1}}
\end{cases}
&  L^2(\R_+ \times \R) &  L^2(\R_+ \times \R) \\
\text{ } & \text{ } & \text{ } & \text{ }  & \text{ }\\
\text{III},\alpha \in [0,\infty) 
&
{\left( r'\right)}^{\frac 12}f_c(r',\theta')
& 
\begin{cases}
r'=r \\
\theta' =\theta -\alpha \log r
\end{cases}
&  L^2(\R^2)  &  L^2(\R_+ \times \T)\\
\text{ } & \text{ } & \text{ } & \text{ }  & \text{ }\\
\text{IV},\alpha \in [0,\infty) 
&
{\left( r'\right)}^{\frac 12}f_c(r',\theta')
& 
\begin{cases}
r'=r \\
\theta' =\theta -\alpha \log r
\end{cases}
&  L^2(\R_+ \times \R)  &  L^2(\R_+ \times \R)
\end{matrix}
\end{equation*}
$$
\text{Table 2}
$$

In Table 2 column $U$ describes unitary maps, and $f_c$ expresses $f$ in new coordinates. In cases III, IV $r, \theta$
represent standard polar and hyperbolic polar coordinates, respectively.

In Section 2 we introduce representations $\mu^{(l)}$, $\mu^{(q)}$, with $\mu^{(l)}$ allowing direct adaptations of Shannon lifts results and  constructions  of \cite{NoPa} to the current setup, where we make a transition to the Fourier transform domain, and we restrict to positive frequencies, and with $\mu^{(q)}$ allowing their further transfer to the context of $Sp(2,\R)$, in the case $\text{I}, \alpha = -1$ of Table 1. The case $\text{I}, \alpha = -1$ allows still  further transfers to all the $\mu^{\mathcal J}$ cases, described in Table 1. In Section 3 we introduce and study the unitary maps of Table 2 allowing reductions of all currently known reproducing formulae of $L^2(\R^2)$ with two dimensional parameterizations 
 of the cases of $\mu^{\mathcal J}$ to the case $\text{I}, \alpha = -1$.

The book by  F\"uhr \cite{Fuh} approaches constructions of wavelet type expansions via powerful tools of representation theory. The existence of a generating function for $\mu^{(l)}$ is guaranteed by the general theory of wavelet transforms developed there. 
It is enough to observe that $\mu^{(l)}$ is  unitarily equivalent to
a countably infinite multiple of the standard, square-integrable representation $\sigma$ of the $ax+b$
group acting on $L^2(\R_+)$, i.e. $\mu^{(l)} = \oplus_{n\in\N}\sigma_n$.  The existence of a generating function or an admissible vector (as it is called in \cite{Fuh}) for $\mu^{(l)}$ follows from Corollary 4.27 of \cite{Fuh}, since each $\sigma_n$ has an admissible  vector and the  $ax+b$ group is type I and nonunimodular.

The paper by Aronszajn discusses the origins of the theory of reproducing kernels. The book by Ali-Antoine-Gazeau \cite{AAG} presents both the current stage of development of the theory of reproducing formulae, as well as the background results. Our current results follow the approaches of  De Mari-Nowak \cite{DMNo1}, \cite{DMNo2}, Cordero-De Mari-Nowak-Tabacco \cite{CDMNT1}, \cite{CDMNT2}, \cite{CDMNT3}, De Mari-De Vito and collaborators \cite{DM&Co1}, \cite{DM&Co2}, \cite{DM&Co3}, \cite{DM&Co4}, Cordero-Tabacco \cite{CoTa}.
The books by Daubechies \cite{Dau}, Gr\"ochenig \cite{Gro}, Folland \cite{Fol1}, Wojtaszczyk \cite{Woj} are comprehensive references on phase space analysis and wavelets. We refer the reader to books by {\L}ojasiewicz \cite{Loj},  Rudin \cite{Rud}, and Folland \cite{Fol2} for the background results we use in our proofs.

\section{ Expansions via one-dimensional affine and wavelet lattice actions}

We define two reference representations $\mu^{(l)}$, $\mu^{(q)}$. Letters $l$, $q$ stand for linear and quadratic oscillations respectively. Let $(Y,\kappa)$ be a measure space equipped with a non-negative,  complete, $\sigma$-finite measure $\kappa$. Let ${\mathcal{H}}$ denote the space $L^2(\R_+ \times Y, dx \times d\kappa(y))$, where $dx$ is the Lebesgue measure on $\R_+$, and $dx \times d\kappa(y)$ is the completion of the product measure defined on $\R_+ \times Y$. For $f\in {\mathcal{H}}$ we define
\begin{align}
\mu^{(l)}_{(u,s)}f(\xi,y)&=s^{1/2}f(s\xi,y)e^{2\pi i u \xi}, s>0, u \in \R,
\label{mu_l}\\
\mu^{(q)}_{(v,t)}f(r,y)&=t^{1/2}f(tr,y)e^{\pi i v r^2}, t>0, v\in \R.
\label{mu_q}
\end{align}
Maps $(u,s)\mapsto \mu^{(l)}_{(u,s)}$, $(v,t)\mapsto \mu^{(q)}_{(v,t)}$ are unitary representations on ${\mathcal{H}}$ of groups $G^{(l)}=\left\{ (u,s)\,|\,u \in \R, s>0 \right\}$, with the composition rule $(u ',s')\circ(u,s)=\left( s'u + u',s's\right)$ and the left Haar measure $\frac{du\,ds}{s^2}$,  and $G^{(q)}=\left\{ (v,t)\,|\,v \in \R, t>0 \right\}$, with the composition rule $(v ',t')\circ(v,t)=\left( (t')^2 v + v',t't \right)$ and the left Haar measure $\frac{dv\,dt}{t^3}$. Representation (\ref{mu_l}) is the standard one-dimensional wavelet action applied to the first coordinate, represented in the frequency domain, and restricted to positive frequencies. Representation (\ref{mu_q}) is an adaptation of (\ref{mu_l}) to the context of the projective metaplectic representation of $Sp(2,\R)$, where only quadratic oscillations occur. The lattice $\Lambda ^{(l)}=\left\{(2^km,2^{k})\right\}_{k,m \in \Z}$ properly discretizes $\mu^{(l)}$, $G^{(l)}$, and $\Lambda ^{(q)}=\left\{(2^{k+1}m,2^{k/2})\right\}_{k,m \in \Z}$ is its appropriate adaptation to $\mu^{(q)}$, $G^{(q)}$.
\begin{prop}\label{l_q_equivalence}
Let $U:{\mathcal{H}}\rightarrow {\mathcal{H}}$ be defined as $ Uf(r,y)=(2r)^{1/2}f(r^2,y)$.

\noindent
(i). Operator $U$ is unitary, and it intertwines representations $\mu^{(l)}$ and $\mu^{(q)}$, 
\begin{equation}
\mu^{(l)}_{(u,s)}=U^*\mu^{(q)}_{\left(2u,s^{1/2}\right)}U.
\label{l_q_equivalence_1}
\end{equation}
(ii). Let $\psi \in {\mathcal{H}}$. The system $\left\{\mu^{(l)}_{(u,s)}\psi \right \}_{\substack{u\in \R \\ s>0}}\subset {\mathcal{H}}$,
with the parameter measure $\frac{du\,ds}{s^2}$, is reproducing if and only if the system $\left\{\mu^{(q)}_{(v,t)}U\psi \right \}_{\substack{v\in \R \\ t>0}}\subset {\mathcal{H}}$,
with the parameter measure $\frac{dv\,dt}{t^3}$, is reproducing.

\noindent
(iii).  Let $\psi \in {\mathcal{H}}$. The system $\left\{\mu^{(l)}_\lambda\psi \right \}_{\lambda \in \Lambda^{(l)}}\subset {\mathcal{H}}$ is reproducing if and only if the system $\left\{\mu^{(q)}_\lambda U\psi \right \}_{\lambda \in \Lambda^{(q)}}\subset {\mathcal{H}}$ is reproducing. In both cases the parameter measure is the counting measure.
\end{prop}

\begin{proof} We start with the proof of (i). A direct calculation shows that $U^{-1}f(r,y)=\left(2r^{1/2}\right)^{-1/2}
f(r^{1/2},y)$. Simple changes of variables verify that both $U$ and $U^{-1}$ preserve the norm of $ {\mathcal{H}}$. For 
$f,g \in {\mathcal{H}}$ we have
\begin{align}
\left\langle f, \mu^{(l)}_{(u,s)}g\right\rangle &= \int_{\R_+\times Y}f(\xi,y)s^{1/2}\overline{g}(s\xi,y)e^{-2\pi i u\xi}d\xi \,d\kappa(y)
\nonumber
\\ &=\int_{\R_+ \times Y}f(r^2,y)s^{1/2}\overline{g}(sr^2,y)e^{-2\pi i u r^2}2rdr\,d\kappa(y)
\nonumber
\\ &=\int_{\R_+ \times Y}\left( 2r\right)^{1/2}f(r^2,y)t^{1/2}\left( 2tr\right)^{1/2}\overline{g}((tr)^2,y)e^{-\pi i v r^2}dr\,d\kappa(y)
\nonumber
\\ &=\left\langle Uf, \mu^{(q)}_{(v,t)}Ug\right\rangle
=\left\langle f, U^{*}\mu^{(q)}_{(v,t)}Ug\right\rangle,
\nonumber
\end{align}
where we have substituted $\xi$ by $r^2$, $s$ by $t^2$, and $2u$ by $v$. 
From the formula above we obtain 
$$
\mu^{(l)}_{(u,s)}= U^{*}\mu^{(q)}_{(v,t)}U,
$$
and this finishes the proof of (i). We apply (i) in order to prove (ii). Substitutions $\frac{v}{2}$ for $u$, and $t^2$ for $s$ give
$$
\int_{\R_+ \times \R}\left| \left\langle f, \mu^{(l)}_{(u,s)}\psi\right\rangle\right|^2 \frac{du\,ds}{s^2}=
\int_{\R_+ \times \R}\left| \left\langle Uf, \mu^{(q)}_{(v,t)}U\psi\right\rangle\right|^2 \frac{dv\,dt}{t^3}.
$$
Polarization formula and the fact that $U$ is unitary allow us to conclude (ii). The proof of (iii) follows in the same way as (ii), with integrals substituted by sums. 
\end{proof}

\begin{thm}\label{Reproducing_Formula_iff_1d_Isometry}
Let us consider $\psi \in \mathcal{H}$. The system
$$
\left\{\mu^{(l)}_{(u,s)}\psi\right \}_{\substack{u\in \R \\ s>0}} \subset \mathcal{H}, 
$$
with the parameter measure $\frac{du\,ds}{s^2}$, is reproducing if and only if the map
$$
g\mapsto \int_Y \overline{\psi\left(s,y\right)} g(y)\,d\kappa(y),
$$
from  $L^2(Y,d\kappa(y))$ into $L^2(\R_+,\frac{ds}{s})$, preserves inner products.
\end{thm}

\begin{proof}
Let $\R^2_+=\R\times \R_+$. For $f,g \in \mathcal{H}$,
in the first step, we express the inner products as iterated integrals. We obtain
\begin{align}
\int_{\R^2_+}\langle f,\mu_{(u,s)}^{(l)}\psi \rangle &\langle \mu_{(u,s)}^{(l)}\psi ,g \rangle \frac{du\,ds}{s^2}=\nonumber\\ 
=\int_{\R^2_+} &\left( \int_{\R_+} \left( \int_Y f(\xi,y) \overline{\psi(s\xi,y)}\,d\kappa(y)\right) \, e^{-2\pi i u\xi } d\xi \right)
\nonumber \\
&\left( \int_{\R_+} \left( \int_Y \overline{g(\xi,y)}\psi(s\xi,y)\,d\kappa(y)\right) \,e^{2\pi i u\xi} d\xi \right) \,\frac{du\,ds}{s}.
\label{Reproducing_Formula_iff_1d_Isometry_step_1}
\end{align}
 Representation of the inner products as iterated integrals is justified by the fact, that for $u,s$ fixed,  functions $f(\xi,y)\, \mu_{(u,s)}^{(l)}\psi (\xi,y)$,  $g(\xi,y)\, \mu_{(u,s)}^{(l)}\psi (\xi,y)$are integrable with respect to $d\xi\times d\kappa (y)$. In the second step, we observe that via polarization we may represent 
the right hand side of (\ref{Reproducing_Formula_iff_1d_Isometry_step_1})
as a sum of expressions of the same form as in (\ref{Reproducing_Formula_iff_1d_Isometry_step_1})
with $f=g$. The resulting terms are non negative, therefore we are allowed to change the outer integral over $\R_+^2$ into an iterative form, with the integration with respect to $u$ performed internally and with respect to $s$ externally. Polarization performed backwards gives us the right hand side of  (\ref{Reproducing_Formula_iff_1d_Isometry_step_1}) with the outer integral over $\R_+^2$ in its iterative form.

We assume now that $f(x,y)=f_1(x)\,f_2(y)$,  $g(x,y)=g_1(x)\,g_2(y)$, with $f_1,g_1 
\in L^1\cap L^\infty$  on $\R_+$, and $f_2,g_2 \in L^1\cap L^\infty$ on $Y$. 
We are allowed to apply Parseval's formula with respect to $u$ and formula 
(\ref{Reproducing_Formula_iff_1d_Isometry_step_1}) becomes
\begin{align}
\int_{\R^2_+}\langle f,\mu_{(u,s)}^{(l)}\psi \rangle &\langle \mu_{(u,s)}^{(l)}\psi ,g \rangle \frac{du\,ds}{s^2}=\nonumber\\ 
=
\int_{\R_+}\Bigg( \int_{\R_+} &\left( \int_Y f_1(\xi)f_2( y))\overline{\psi(s\xi, y)}\,d\kappa(y)\right)
\nonumber \\
&\left( \int_Y\overline{ g_1(\xi) g_2(y)}\psi(s\xi, y)\,d\kappa(y)\right) \, d\xi \Bigg) \frac{ds}{s}.
\label{Reproducing_Formula_iff_1d_Isometry_step_2}
\end{align}
 In the third step we change the order of integration with respect to $\xi$ and $s$ and we apply multiplicative invariance of measure $\frac{ds}{s}$. Formula (\ref{Reproducing_Formula_iff_1d_Isometry_step_2}) becomes
\begin{align}
\int_{\R^2_+}\langle f,\mu_{(u,s)}^{(l)}\psi \rangle \langle \mu_{(u,s)}^{(l)}\psi ,g \rangle \frac{du\,ds}{s^2}&=\nonumber\\ 
= \int_{\R_+}f_1(\xi)\overline{g_1(\xi)}
\Bigg( \int_{\R_+} &\left( \int_Y f_2(y)\overline{\psi( s,y)}\,d\kappa(y) \right)
\nonumber \\
&\left( \int_Y \overline{g_2(y)}\psi( s,y)\,d\kappa(y)\right) \, \frac{ds}{s}\Bigg) \,d\xi.
\label{reproduction_for_tensors}
\end{align}
Change of the order of integration is again justified via polarization. We write the expressions under the integral signs as sums of the form we obtain for $f_1=g_1$, $f_2=g_2$. Non-negativity of the resulting terms allows us to apply Fubini's theorem.

If the system $\left\{\mu_{(u,s)}^{(l)}\psi \right\}_{u\in \R, s>0}$ is reproducing, then, via formula 
(\ref{reproduction_for_tensors}), we conclude that the maps 
$f\mapsto \int_Y \overline{\psi\left( s,y\right)} f(y)\,d\kappa(y)$ restricted to $f\in L^1\cap L^\infty$  preserve inner products. A standard density argument allows us to extend the statement to all $f\in L^2(Y, d\kappa(y))$. Conversely, if the maps $f\mapsto \int_Y \overline{\psi\left( s,y\right)} f(y)\,d\kappa(y)$  preserve inner products, then 
(\ref{reproduction_for_tensors}) allows us to conclude that for $f,g \in \mathcal{H}$ being finite sums of tensor products of the form $f_1(x)\,f_2(y)$,  $g_1(x)\,g_2(y)$, with $f_1,f_2,g_1,g_2\in L^1\cap L^\infty$ we have
$$
\int_{\R^2_+}\langle f,\mu_{(u,s)}^{(l)}\psi \rangle \langle \mu_{(u,s)}^{(l)}\psi ,g \rangle \frac{du\,ds}{s^2}
=\langle f,g \rangle.
$$ 
Again, a standard density argument allows us to extend the equality to all $f,g\in \mathcal{H}$.
\end{proof}

\begin{cor}\label{Reproducing_Formula_iff_1d_Isometry_q_case}
Let us consider $\psi \in \mathcal{H}$. The system
$$
\left\{\mu^{(q)}_{(v,t)}\psi\right \}_{\substack{v\in \R \\ t>0}} \subset \mathcal{H}, 
$$
with the parameter measure $\frac{dv\,dt}{t^3}$, is reproducing if and only if the map
$$
g\mapsto \int_Y \overline{\psi\left(r,y\right)} g(y)\,d\kappa(y),
$$
from  $L^2(Y,d\kappa(y))$ into $L^2(\R_+,\frac{dr}{r^2})$, preserves inner products.
\end{cor}

\begin{proof} The result is a direct consequence of Proposition \ref{l_q_equivalence} (ii) and Theorem \ref{Reproducing_Formula_iff_1d_Isometry}.
\end{proof}

For a measurable function $f$, defined on a topological space $X$, equipped with a Borel measure $\nu$, we define its essential support $\text{ess-supp}\,f$ as the intersection of all closed sets $F$, satisfying 
$f(x)=0$ for $\nu$-almost every $x$ in the complement of $F$.  
We will need the following standard representation of the inner product on $L^2(\R)$, valid for band limited functions (see e.g. Lemma 2.1 in \cite{NoPa})

\begin{lem}\label{Inner_Product_for_BL} Let us suppose that for $f,g\in L^2(\R)$ we have $\text{ess-supp}\,f, g\subset 
[0, 2^{-k}]$. Then
$$
\int_{0}^{2^{-k}}f(\xi)\overline{g(\xi)}\,
d\xi = 2^k\sum_{m\in \Z} \hat{f}\left(2^km\right) \overline{ \hat{g}\left(2^km\right)}.
$$
\end{lem}

\begin{thm}\label{Reproducing_Discrete_System_iff_1d_Isometry}
 Let  us consider $\psi \in \mathcal{H}$. Suppose that for almost every $y \in Y$ $\text{ess-supp}\, \psi(\cdot,y)\subset [0,1]$.
The system
$$
\left\{\mu^{(l)}_{\lambda}\psi\right \}_{\lambda \in \Lambda^{(l)}} \subset \mathcal{H}, 
$$
with the parameter measure being the counting measure on $\Lambda^{(l)}$,
is reproducing if and only if for 
every pair $f,g\in L^2(Y,d\kappa (y))$ the equality
\begin{equation}
\langle f,g \rangle = \sum_k \left( \int_Y
\overline{\psi\left(2^k \xi,y\right)} f(y)\,d\kappa(y)\right)
\left( \int_Y \psi\left(2^k \xi,y\right) \overline{g(y)}\,d\kappa(y)\right)
\label{discrete_isometry}
\end{equation}
holds for almost every $\xi\in \R_+$. 
\end{thm}

\begin{proof}
In the first step we express the inner products as iterated integrals. We obtain
\begin{align}
\sum_{\lambda \in \Lambda^{(l)}}\langle f,\mu_{\lambda}^{(l)}\psi \rangle &\langle \mu_{\lambda}^{(l)}\psi ,g \rangle =\nonumber\\ 
=\sum_{k,m\in \Z} &\left( \int_{\R_+} \left(\int_Y f(\xi,y) 2^{k/2}\overline{\psi(2^k\xi,y)}\,d\kappa(y) \right)\, e^{-2\pi i 2^k m\xi } d\xi \right)
\nonumber \\
&\left(\int_{\R_+} \left( \int_Y \overline{g(\xi,y)}2^{k/2}\psi(2^k\xi,y)\,d\kappa(y) \right) \,e^{2\pi i 2^k m\xi} d\xi \right).
\label{Tight_Frame_iff_1d_Isometry_step_1}
\end{align}
 Representation of the inner products as iterated integrals is justified by the fact, that for $\lambda$ fixed,  functions $f(\xi,y)\, \mu_{\lambda}^{(l)}\psi (\xi,y)$,  $g(\xi,y)\, \mu_{\lambda}^{(l)}\psi (\xi,y)$are integrable with respect to $d\xi\times d\kappa (y)$. In the second step, we observe that via polarization we may represent 
the right hand side of (\ref{Tight_Frame_iff_1d_Isometry_step_1})
as a sum of expressions of the same form as in (\ref{Tight_Frame_iff_1d_Isometry_step_1})
with $f=g$. The resulting terms are non negative, therefore we are allowed to change the summation over $k,m\in \Z$ into an iterative form, with the summation with respect to $m$ performed internally and with respect to $k$ externally. Polarization performed backwards gives us the right hand side of  (\ref{Tight_Frame_iff_1d_Isometry_step_1}) with the the summation over $k,m\in \Z$ in its iterative form.

We assume now that $f(x,y)=f_1(x)\,f_2(y)$,  $g(x,y)=g_1(x)\,g_2(y)$, with $f_1,g_1 
\in L^1\cap L^\infty$  on $\R_+$, and $f_2,g_2 \in L^1\cap L^\infty$ on $Y$. 
We are allowed to apply  Lemma \ref{Inner_Product_for_BL} and formula 
(\ref{Tight_Frame_iff_1d_Isometry_step_1}) becomes
\begin{align}
\sum_{\lambda \in \Lambda^{(l)}}\langle f,\mu_{\lambda}^{(l)}\psi \rangle &\langle \mu_{\lambda}^{(l)}\psi ,g \rangle =\nonumber\\ 
=
\sum_{k\in \Z}\Bigg(\int_{\R_+} &\left(\int_Y f_1(\xi)f_2( y))\overline{\psi(2^k\xi, y)}\,d\kappa(y)\right)
\nonumber \\
&\left( \int_Y\overline{ g_1(\xi) g_2(y)}\psi(2^k\xi, y)\,d\kappa(y)\right)\, d\xi \Bigg).
\label{Tight_Frame_iff_1d_Isometry_step_2}
\end{align}
 The usage of Lemma \ref{Inner_Product_for_BL} is justified by the fact that for almost every $y\in Y$ we have $\text{ess-supp}\,\psi({2^k}\cdot,y)\subset \left[0, 2^{-k}\right]$. We represent the square integrable kernel $\overline{\psi({2^k}\cdot,\cdot)}$, defined on $\R_+\times Y$, $k$ is fixed, as an infinite sum of orthogonal tensor products of functions, band limited, with respect to the first coordinate, and square integrable, with respect to the second coordinate. Then, we apply Lemma \ref{Inner_Product_for_BL} to finite sums, and next we pass to norm limits in both expressions, the original one, and the one obtained by an application of  Lemma \ref{Inner_Product_for_BL}. 
In the third step we change the order of integration with respect to $\xi$ and summation with respect to $k$. Formula (\ref{Tight_Frame_iff_1d_Isometry_step_2}) becomes
\begin{align}
\sum_{\lambda \in \Lambda^{(l)}}\langle f,\mu_{\lambda}^{(l)}\psi \rangle \langle \mu_{\lambda}^{(l)}\psi ,g \rangle &=\nonumber\\ 
= \int_{\R_+}f_1(\xi)\overline{g_1(\xi)}
\Bigg(\sum_{k\in \Z} &\left( \int_Y f_2(y)\overline{\psi( 2^k\xi,y)}\,d\kappa(y)\right)
\nonumber \\
&\left(\int_Y \overline{g_2(y)}\psi( 2^k\xi,y)\,d\kappa(y)\right)\Bigg)\,d\xi.
\label{discrete_reproduction_for_tensors}
\end{align}
Change of the order of integration and summation is again justified via polarization. We rewrite the expressions as sums of the form we obtain for $f_1=g_1$, $f_2=g_2$. Non-negativity of the resulting terms allows us to apply Fubini's theorem.

If the system $\left\{\mu^{(l)}_{\lambda}\psi\right \}_{\lambda \in \Lambda^{(l)}}$ is reproducing, then, via formula 
(\ref{discrete_reproduction_for_tensors}), we have  
$$\int_\R|\hat{f}_1(\xi)|^2\sum_k \left|\int_\R f_2(x_2) 
\overline{\psi(\widehat{2^k \xi},x_2)}\,dx_2\,\right|^2d\xi
=||f_1||^2||f_2||^2,
$$
for all $f_1,f_2\in \mathcal{S}(\R)$,
and therefore for every $f\in \mathcal{S}(\R)$
\begin{equation}
\sum_k \left|\int_\R 
\overline{\psi(\widehat{2^k \xi},y)} f(y)\,dy\,\right|^2 = ||f||^2
\label{discrete_a_e_norm_condition}
\end{equation}
for almost every $\xi \in \R$. A standard density argument, making use of the convergence in the mixed norm space $L^\infty(l^2)$, allows us to conclude that for every 
$f\in L^2(\R)$
(\ref{discrete_a_e_norm_condition})
holds for almost every $\xi \in \R$. The fact that 
for every pair $f,g\in L^2(\R)$ the equality (\ref{discrete_isometry})
holds for almost every $\xi\in \R$ follows by polarization. 
Conversely, if for every pair $f,g\in L^2(\R)$ the equality (\ref{discrete_isometry})
holds for almost every $\xi\in \R$, then
(\ref{discrete_reproduction_for_tensors}) allows us to conclude that for $f,g \in L^2(\R^2)$ being finite sums of tensor products of the form $f_1(x_1)\,f_2(x_2)$,  $g_1(x_1)\,g_2(x_2)$, with $f_1,f_2,g_1,g_2\in \mathcal{S}(\R)$ we have
$$
\sum_{k,m}\langle f,\psi_{k,m} \rangle \langle \psi_{k,m},g \rangle =\langle f,g\rangle.
$$
Again, a standard density argument allows us to extend the equality to all $f,g\in L^2(\R^2)$.
\end{proof}

\begin{cor}\label{Reproducing_Discrete_System_iff_1d_Isometry_q_case}
 Let us consider $\psi \in \mathcal{H}$. Suppose that for almost every $y \in Y$ $\text{ess-supp}\, \psi(\cdot,y)\subset [0,1]$.
The system
$$
\left\{\mu^{(q)}_{\lambda}\psi\right \}_{\lambda \in \Lambda^{(q)}} \subset \mathcal{H}, 
$$
with the parameter measure being the counting measure on $\Lambda^{(q)}$,
is reproducing if and only if for 
every pair $f,g\in L^2(Y,d\kappa (y))$ the equality
\begin{equation}
\langle f,g \rangle = \frac{1}{2r}\sum_k 2^{-k/2}\left( \int_Y
\overline{\psi\left(2^{k/2} \, r,y\right)} f(y)\,d\kappa(y)\right)
\left(\int_Y \psi\left(2^{k/2} \, r,y\right) \overline{g(y)}\,d\kappa(y)\right)
\nonumber
\end{equation}
holds for almost every $r\in \R_+$. 
\end{cor}

\begin{proof} The result is a direct consequence of Proposition \ref{l_q_equivalence} (iii) and Theorem \ref{Reproducing_Discrete_System_iff_1d_Isometry}.
\end{proof}

Let $\psi ^S(\xi)=\chi_{(1,2]}(\xi)$. The system $\left\{ \psi^S_{\lambda} \right\}_{\lambda \in \Lambda^{(l)}}\subset L^2(\R_+)$, with 
$$\psi^S_\lambda (\xi)=2^{k/2}\chi_{(1,2]}\left(2^{k}\xi\right)e^{2\pi i 2^km\xi},
$$
is the Shannon wavelet system adapted to $\R_+$, with $L^2(\R_+)$ representing the Fourier transform domain. It is a family of standard trigonometric systems adapted to the dyadic partition of $\R_+$. We describe its lifts to
$L^2(\R_+ \times Y)$ , with $Y$ being $\R$ and $\T$, lifts adapting the constructions done in \cite{NoPa} to the current context.  First, we do it for $\mu^{(l)}$, and then we transfer the resulting systems to $\mu^{(q)}$ via the unitary map of Proposition \ref{l_q_equivalence}.

We move now to $L^2(\R_+\times Y)$. We introduce $e_{k,l} (y)=\chi_{(k,k+1]}(y)\,e^{2\pi ily}$, with $k,l\in \Z$, and $f_m (s)=\chi_{(2^{-m},2^{-m+1} ]} (s)$, with $m \ge 1, m\in \Z$, i.e. $m\in \N$. The system $\left\{e_{k,l} \right\}_{k,l\in \Z}$ is an orthonormal basis of $L^2 (\R)$, and $\left\{c_f^{-1} f_m \right\}_{m\ge 1}$, where $c_f=(\log 2)^{1/2}$, is an orthonormal system of 
$L^2 (\R_+,\frac{ds}{s})$. Let $D_\R:\Z \times \Z \rightarrow \N$, $D_\T:\Z \rightarrow \N$ be two bijections. We define the corresponding  generating functions $\psi^{D_\R} \in L^2 (\R_+\times \R)$,  $\psi^{D_\T} \in L^2 (\R_+\times \T)$ as 
\begin{equation}
\psi^{D_\R}(\xi,y)=\sum_{k,l\in \Z} f_{D_\R(k,l)}(\xi)e_{k,l} (y),
\label{2d_generating_function_R}
\end{equation}
\begin{equation}
\psi^{D_\T}(\xi,y)=\sum_{l\in \Z} f_{D_\T(l)}(\xi)e_{0,l} (y).
\label{2d_generating_function_T}
\end{equation}

The following two lemmas summarize the basic properties of the generating functions $\psi^{D_\R}$, $\psi^{D_\T}$. Their
proofs follow exactly the steps of the proof of Lemma 2.2 of \cite{NoPa}, and are omitted from the current presentation.

\begin{lem}\label{Basic_psi_^D_R_properties}
Let $\psi^{D_\R}$ be the generating function defined in
(\ref{2d_generating_function_R}). Then
\begin{align}
\text{(i)}\, &\text{the sum } (\ref{2d_generating_function_R}) \text{ representing } \psi^{D_\R}(\xi,y)\, \text{ consists of a single term } \nonumber \\
&f_{D_\R(k,l)}(\xi)e_{k,l}(y), 
\text{for } \xi\in (0,1] \text{, with the unique }k,l 
\text{ satisfying } 
\nonumber \\
&\xi \in (2^{-D_\R(k,l)}, 2^{-D_\R(k,l)+1}], \text{ and it contains no non-zero terms for }
\nonumber \\
&\xi\notin  (0,1],
\nonumber \\
\text{(ii)}\, &\text{ess-supp}\,\psi^{D_\R}({\cdot},y)\subset [0,1] \text{ for every }y\in \R ,
\nonumber \\
\text{(iii)}\, &\int_{\R_+ \times \R}\left| \psi^{D_\R}(\xi,y) \right|^2dy\, d\xi =1,
\nonumber \\
\text{(iv)}\,&S_N^{D_\R}(\xi,y)=\sum_{|k|\le N,|l|\le N}f_{D_\R(k,l)}(\xi)e_{k,l}(y)
\text{ converges to }\psi^{D_\R}(\xi,y) 
\nonumber \\
&\text{ in } L^2(\R_+ \times \R) \text{, as }
N \rightarrow \infty.
\nonumber
\end{align}
\end{lem} 

\begin{lem}\label{Basic_psi_^D_T_properties} 
Let $\psi^{D_\T}$ be the generating function defined in
(\ref{2d_generating_function_T}). Then
\begin{align}
\text{(i)}\, &\text{the sum } (\ref{2d_generating_function_T}) \text{ representing } \psi^{D_\T}(\xi,y)\, \text{ consists of a single term } \nonumber \\
&f_{D_\T(l)}(\xi)e_{0,l}(y), 
\text{for } \xi\in (0,1] \text{, with the unique }l 
\text{ satisfying } 
\nonumber \\
&\xi \in (2^{-D_\T(l)}, 2^{-D_\T(l)+1}], \text{ and it contains no non-zero terms for }
\nonumber \\
&\xi\notin  (0,1],
\nonumber \\
\text{(ii)}\, &\text{ess-supp}\,\psi^{D_\T}({\cdot},y)\subset [0,1] \text{ for every }y\in \R ,
\nonumber \\
\text{(iii)}\, &\int_{\R_+ \times \T}\left| \psi^{D_\T}(\xi,y) \right|^2dy\, d\xi =1,
\nonumber \\
\text{(iv)}\,&S_N^{D_\T}(\xi,y)=\sum_{|l|\le N}f_{D_\T(l)}(\xi)e_{0,l}(y)
\text{ converges to }\psi^{D_\T}(\xi,y) 
\nonumber \\
&\text{ in } L^2(\R_+ \times \T) \text{, as }
N \rightarrow \infty.
\nonumber
\end{align}
\end{lem} 

\begin{thm}\label{2d_Reproducing_Systems_D_R_D_T}
Systems $\left\{c_f^{-1}\mu^{(l)}_{(u,s)}\psi^{D_\R}\right\}_{u\in \R,s>0}$,  $\left\{c_f^{-1}\mu^{(l)}_{(u,s)}\psi^{D_\T}\right\}_{u\in \R,s>0}$, with generating functions $\psi^{D_\R}$, $\psi^{D_\T}$ defined in  (\ref{2d_generating_function_R}), (\ref{2d_generating_function_T}), both with the same parameter measure $\frac{du\,ds}{s^2}$, are reproducing in $L^2(\R_+\times \R)$,  $L^2(\R_+\times \T)$, respectively.
\end{thm}
\begin{proof} Both proofs, based on Theorem \ref{Reproducing_Formula_iff_1d_Isometry}, follow the steps of the proof of Corollary 1.2 of \cite{NoPa}, with the adjustments indicated in Lemmas
\ref{Basic_psi_^D_R_properties}, \ref{Basic_psi_^D_T_properties}, respectively.
\end{proof}

\begin{thm}\label{2d_Orthonormal_Bases_D_R_D_T}
Systems $\left\{\mu^{(l)}_{\lambda}\psi^{D_\R}\right\}_{\lambda \in \Lambda ^{(l)}}$,   $\left\{\mu^{(l)}_{\lambda}\psi^{D_\T}\right\}_{\lambda \in \Lambda ^{(l)}}$,  with $\psi^{D_\R}$, $\psi^{D_\T}$ defined in  (\ref{2d_generating_function_R}), (\ref{2d_generating_function_T}), are orthonormal bases of $L^2(\R_+\times \R)$,  $L^2(\R_+\times \T)$, respectively.
\end{thm}
\begin{proof}  Both proofs, based on Theorem \ref{Reproducing_Discrete_System_iff_1d_Isometry}, follow the steps of the proof of Corollary 1.4 of \cite{NoPa}, with the adjustments indicated in Lemmas
\ref{Basic_psi_^D_R_properties}, \ref{Basic_psi_^D_T_properties}, respectively.
\end{proof}

\begin{cor}\label{2d_Reproducing_Systems_D_R_D_T_q}
Systems $\left\{c_f^{-1}\mu^{(q)}_{(v,t)}U\psi^{D_\R}\right\}_{v\in \R,t>0}$,  $\left\{c_f^{-1}\mu^{(q)}_{(v,t)}U\psi^{D_\T}\right\}_{v\in \R,t>0}$, with generating functions $U\psi^{D_\R}$, $U\psi^{D_\T}$, defined via an application of the unitary map $U$ of  Proposition \ref{l_q_equivalence} to functions (\ref{2d_generating_function_R}), (\ref{2d_generating_function_T}), both systems with the same parameter measure $\frac{dv\,dt}{t^3}$, are reproducing in $L^2(\R_+\times \R)$,  $L^2(\R_+\times \T)$, respectively.
\end{cor}
\begin{proof} Both proofs follow directly out of Theorem \ref{2d_Reproducing_Systems_D_R_D_T} and Proposition \ref{l_q_equivalence} (ii).
\end{proof}

\begin{cor}\label{2d_Orthonormal_Bases_D_R_D_T_q}
Systems $\left\{\mu^{(q)}_{\lambda}U\psi^{D_\R}\right\}_{\lambda \in \Lambda ^{(q)}}$,   $\left\{\mu^{(q)}_{\lambda}U\psi^{D_\T}\right\}_{\lambda \in \Lambda ^{(q)}}$,  with $U\psi^{D_\R}$, $U\psi^{D_\T}$ defined via an application of the unitary map $U$ of  Proposition \ref{l_q_equivalence} to functions (\ref{2d_generating_function_R}), (\ref{2d_generating_function_T}), are orthonormal bases of $L^2(\R_+\times \R)$,  $L^2(\R_+\times \T)$, respectively.
\end{cor}
\begin{proof}  Both proofs follow directly out of Theorem \ref{2d_Orthonormal_Bases_D_R_D_T} and Proposition \ref{l_q_equivalence} (iii).
\end{proof}

\section{Unitary equivalence of restrictions to reproducing subgroups of type $\mathcal{E}_2$}

We list representatives, up to conjugation within $Sp(2,\R)$, of all reproducing formulae obtained out of restrictions of the projective metaplectic representation of $Sp(2,\R)$ to two-dimensional, connected Lie subgroups of $\mathcal{E}_2$. Each such reproducing formula is conjugate to exactly one reproducing formula of the list. All reproducing formulae of the list are non-conjugate. We refer the reader to \cite{DM&Co2}, \cite{DM&Co3} for details and a comprehensive presentation of the topic. For the sake of simplicity, we restrict attention to the cases of single connected components of multiplicative actions on the first coordinate, and we choose $\R_+$ for them. The choice of $\R_-$ can be treated in a similar manner. The transition to two components $\R_+ \cup \R_-$ follows in a standard way, see e.g.  \cite{DMNo1}, \cite{DM&Co3}. In the parametrizations of two-dimensional subgroups of $\mathcal{E}_2$ listed below we use $u,t \in \R$. 

\noindent
{\bf I.} 
We have an additional parameter $\alpha \in [-1,0)$,  the Hilbert space $\mathcal{H}$ is $L^2(\R_+\times \R)$, and the action on $f\in \mathcal{H}$ is given by formula
\begin{equation}
\mu^{I_\alpha}_{(u,t)}f(x_1,x_2)=e^{-(2\alpha+1)t/2}e^{\pi i u x_1^2}f\left( e^{-\alpha t}x_1,e^{-(\alpha +1)t} x_2\right),
\label{mu_I_alpha}
\end{equation}
with the corresponding composition rule 
$$
(u',t')\circ (u,t)=\left( u'+e^{-2\alpha t'}u,t'+t \right)
$$ 
and the left Haar measure $-\alpha \,du\,e^{2\alpha t}dt $.

\noindent
{\bf II.} We do not have additional parameters, the  Hilbert space $\mathcal{H}$ is $L^2(\R_+\times \R)$, and the action on $f\in \mathcal{H}$ is given by formula
\begin{equation}
\mu^{II}_{(u,t)}f(x_1,x_2)=e^{-t}e^{\pi i u x_1^2}f\left(e^{-t}(x_1, x_2-t x_1)\right),
\label{mu_II}
\end{equation}
with the corresponding composition rule 
$$
(u',t')\circ (u,t)=\left( u'+e^{-2 t'}u,t'+t \right)
$$
and the left Haar measure $du\,e^{2t}dt $.

In order to describe case III, we introduce standard polar coordinates 
\begin{equation}
\begin{cases}
x_1 =r\cos 2\pi \theta \\
x_2 =r\sin 2\pi \theta
\end{cases},
\label{standard_polar_coordinates}
\end{equation}
$r>0$, $\theta \in [0,1)$, and we interpret the interval $[0,1)$ as the unit circle $\T$.
We define rotations
\begin{equation}
R_\theta=
\left[
\begin{array}{cc}
\cos 2\pi \theta & \sin 2\pi \theta \\
-\sin 2\pi \theta & \cos 2\pi \theta
\end{array}
\right].
\nonumber
\end{equation}
For a function $f\in L^2(\R^2)$, we write $f_{p}$ for its representation in polar coordinates, i.e. $f_{p}(r,\theta)=f(x_1,x_2)$.

\noindent
{\bf III.}  In this case the additional parameter is $\alpha \in [0,\infty)$. The Hilbert space $\mathcal{H}$ is $L^2(\R^2)$, and the action of the representation on $f\in \mathcal{H}$, is
\begin{equation}
\mu^{III_\alpha}_{(u,t)}f(x_1,x_2)=e^{-t}e^{\pi i u \left( x_1^2 + x_2^2\right)}
f\left( e^{-t} R_{-\alpha t}(x_1, x_2) \right),
\label{mu_III_alpha}
\end{equation}
with the corresponding composition rule 
$$
(u',t')\circ (u,t)=\left( u'+e^{-2 t'}u,t'+t \right)
$$ 
and the left Haar measure $du\,e^{2t}dt$.

In order to describe case IV, we introduce hyperbolic  polar coordinates 
\begin{equation}
\begin{cases}
x_1 =r\cosh \theta \\
x_2 =r\sinh \theta
\end{cases},
\label{hyperbolic_polar_coordinates}
\end{equation}
$r,\theta \in \R$, and hyperbolic rotations
\begin{equation}
A_\theta=
\left[
\begin{array}{cc}
\cosh \theta & \sinh \theta \\
\sinh \theta & \cosh \theta
\end{array}
\right].
\nonumber
\end{equation}
For a function $f\in L^2(\R^2)$, we write $f_{h}$ for its representation in hyperbolic polar coordinates, i.e. $f_{h}(r,\theta)=f(x_1,x_2)$.

\noindent
{\bf IV.}  In this case the additional parameter is $\alpha \in [0,\infty)$. The Hilbert space $\mathcal{H}$ is $L^2(\R_+\times \R)$, and the action of the representation on $f\in \mathcal{H}$, is
\begin{equation}
\mu^{IV_\alpha}_{(u,t)}f(x_1,x_2)=e^{-t}e^{\pi i u \left( x_1^2 - x_2^2\right)}
f\left( e^{-t} A_{-\alpha t}(x_1, x_2) \right),
\label{mu_IV_alpha}
\end{equation}
with the corresponding composition rule 
$$
(u',t')\circ (u,t)=\left( u'+e^{-2 t'}u,t'+t \right)
$$ 
and the left Haar measure $du\,e^{2t}dt$.

In what follows we introduce coordinate systems  needed for the reductions of cases I-IV to $\mu^{(q)}$.

\noindent
{\bf I}, $-1 \le \alpha <0$. 
\begin{equation}
\begin{cases}
y_1 =x_1 \\
y_2 =x_1^{-\frac{\alpha +1}{\alpha}}x_2
\end{cases},
\begin{cases}
x_1 =y_1 \\
x_2 =y_1^{\frac{\alpha +1}{\alpha}}y_2
\end{cases},
\text{ Jacobian} = \frac{\partial x_2}{\partial y_2}=y_1^{\frac{\alpha +1}{\alpha}}.
\label{I_alpha_coordinates}
\end{equation}

\noindent
{\bf II}. 
\begin{equation}
\begin{cases}
y_1 =x_1 \\
y_2 ={\frac{x_2 -x_1\log x_1}{x_1}}
\end{cases},
\begin{cases}
x_1 =y_1 \\
x_2 =y_1y_2 +y_1\log y_1
\end{cases},
\text{ Jacobian} = \frac{\partial x_2}{\partial y_2}=y_1.
\label{II_coordinates}
\end{equation}

\noindent
{\bf III}, $\alpha \ge 0$. 
\begin{equation}
\begin{cases}
r'=r \\
\theta' =\theta -\alpha \log r
\end{cases},
\begin{cases}
r =r' \\
\theta =\theta' +\alpha \log r'
\end{cases},
\text{ Jacobian} = \frac{\partial \theta}{\partial \theta'}=1,
\label{III_alpha_coordinates}
\end{equation}
where $(r,\theta)$ are the standard polar coordinates of (\ref{standard_polar_coordinates}).

\noindent
{\bf IV}, $\alpha \ge 0$. 
\begin{equation}
\begin{cases}
r'=r \\
\theta' =\theta -\alpha \log r
\end{cases},
\begin{cases}
r =r' \\
\theta =\theta' +\alpha \log r'
\end{cases},
\text{ Jacobian} = \frac{\partial \theta}{\partial \theta'}=1,
\label{IV_alpha_coordinates}
\end{equation}
where $(r,\theta)$ are the hyperbolic polar coordinates of (\ref{hyperbolic_polar_coordinates}).

For ${\mathcal J}=I_\alpha,II,III_\alpha,IV_\alpha$ we define the corresponding lattice $\Lambda^{\mathcal J}$ as the image of $\Lambda^{(q)}$ via the inverse of $(u,t)\mapsto (u,e^{-\alpha t})$ for ${\mathcal J}=I_\alpha$, i.e. it is
$\left\{ \left(2^{k+1}m, -\frac{\log 2}{2\alpha}k \right)\right\}_{m,k\in \Z}$, and via the inverse of  $(u,t)\mapsto (u,e^{- t})$ for ${\mathcal J}=II,III_\alpha,IV_\alpha$, i.e. it is $\left\{ \left(2^{k+1}m, -\frac{\log 2}{2}k \right)\right\}_{m,k\in \Z}$.

\begin{thm}\label{I-IV_equivalence}
Let $\mu^{(q)}$ be defined in (\ref{mu_q}), with $Y=\R$ in cases I,II,IV, and $Y=\T$ in case III. In all cases $\kappa$ is the Lebesgue measure.

\noindent
(i). In each case $f_c$ expresses $f$ in the adequate coordinate system. \newline
I. Let us define $U^{I_\alpha}f(y_1,y_2)=y_1^{\frac{\alpha +1}{2\alpha}}f_c(y_1,y_2)$, where $f_c$ is the expression of $f$ in coordinates  (\ref{I_alpha_coordinates}). Then $U^{I_\alpha}$ is unitary and we have the following intertwining property
$$
U^{I_\alpha}\mu^{I_\alpha}_{(u,t)}=\mu^{(q)}_{\left( u,e^{-\alpha t}\right)}U^{I_\alpha}.
$$
II. Let us define $U^{II}f(y_1,y_2)=y_1^{\frac 12}f_c(y_1,y_2)$, where $f_c$ is the expression of $f$ in coordinates  (\ref{II_coordinates}). Then $U^{II}$ is unitary and  we have the following intertwining property
$$
U^{II}\mu^{II}_{(u,t)}=\mu^{(q)}_{\left( u,e^{- t}\right)}U^{II}.
$$
III. Let us define $U^{III_\alpha}f(r',\theta')={\left( r'\right)}^{\frac 12}f_c(r',\theta')$, where $f_c$ is the expression of $f_p$ in coordinates  (\ref{III_alpha_coordinates}), and $f_p$ expresses $f$ in standard polar coordinates (\ref{standard_polar_coordinates}). Then $U^{III_\alpha}$ is unitary and  we have the following intertwining property
$$
U^{III_\alpha}\mu^{III_\alpha}_{(u,t)}=\mu^{(q)}_{\left( u,e^{-t}\right)}U^{III_\alpha}.
$$
IV. Let us define $U^{IV_\alpha}f(r',\theta')={\left( r'\right)}^{\frac 12}f_c(r',\theta')$, where $f_c$ is the expression of $f_h$ in coordinates  (\ref{IV_alpha_coordinates}), and $f_h$ expresses $f$ in hyperbolic polar coordinates (\ref{hyperbolic_polar_coordinates}). Then $U^{IV_\alpha}$ is unitary and  we have the following intertwining property
$$
U^{IV_\alpha}\mu^{IV_\alpha}_{(u,t)}=\mu^{(q)}_{\left( u,e^{-t}\right)}U^{IV_\alpha}.
$$
(ii). In each of the cases ${\mathcal J}=I_\alpha,II,III_\alpha,IV_\alpha$ the system 
$\left\{\mu^{\mathcal J}_{(u,t)}\psi \right \}_{u,t\in \R}$,
with the Hilbert space and  the left Haar measure described in (\ref{mu_I_alpha}), (\ref{mu_II}), (\ref{mu_III_alpha}), (\ref{mu_IV_alpha}), respectively,  is reproducing,  if and only if, the system $\left\{\mu^{(q)}_{(v,s)}U^{\mathcal J}\psi \right \}_
{\substack{v\in \R \\ s>0}}$,
with the parameter measure $\frac{dv\,ds}{s^3}$, is reproducing in $L^2(\R_+ \times Y)$, $Y=\R$ in cases I,II,IV, and $Y=\T$ in case III.

\noindent
(iii). In each of the cases ${\mathcal J}=I_\alpha,II,III_\alpha,IV_\alpha$ the system 
$\left\{\mu^{\mathcal J}_{\lambda}\psi \right \}_{\lambda\in \Lambda^{\mathcal J}}$,
with the Hilbert space described in (\ref{mu_I_alpha}), (\ref{mu_II}), (\ref{mu_III_alpha}), (\ref{mu_IV_alpha}), respectively, is reproducing,  if and only if, the system $\left\{\mu^{(q)}_{\lambda}U^{\mathcal J}\psi \right \}_
{\lambda \in \Lambda^{(q)}}$ is reproducing in $L^2(\R_+ \times Y)$, $Y=\R$ in cases I,II,IV, and $Y=\T$ in case III. In all cases the parameter measure is the counting measure.
\end{thm}

\begin{proof} We begin with the proof of (i). Verification of the fact that operators $U^{\mathcal J}$ are unitary is straightforward in all of the cases. \newline
I. In order to prove the intertwining property of I we substitute $e^t$ by $s$, next we introduce coordinates (\ref{I_alpha_coordinates}), and then we substitute $s^{-\alpha}$ by $r$,
\begin{align}
\left\langle f,\mu^{I_\alpha}_{(u,t)}g \right\rangle &=\int_{\R_+ \times \R}f(x_1,x_2)e^{-(2\alpha +1)t/2}\overline{g}\left(e^{-\alpha t}x_1,e^{-(\alpha +1)t}x_2\right) e^{-\pi i ux_1^2}dx_1\,dx_2 \nonumber\\
&=\int_{\R_+  \times \R}f(x_1,x_2)s^{-\frac{2\alpha +1}{2}}\overline{g}\left(s^{-\alpha }x_1,s^{-(\alpha +1)}x_2\right) e^{-\pi i ux_1^2}dx_1\,dx_2 \nonumber \\
&=\int_{\R_+ \times \R}f_c(y_1,y_2)s^{-\frac{2\alpha +1}{2}}\overline{g_c}\left(s^{-\alpha }y_1,y_2\right) y_1^\frac{\alpha+1}{\alpha}e^{-\pi i uy_1^2}dy_1\,dy_2 \nonumber \\
&=\int_{\R_+ \times \R}y_1^\frac{\alpha+1}{2\alpha}f_c(y_1,y_2)r^{\frac 12}\left(ry_1\right)^\frac{\alpha+1}{2\alpha}\overline{g_c}(ry_1,y_2) e^{-\pi i uy_1^2}dy_1\,dy_2 \nonumber \\
&=\left\langle U^{I_\alpha}f,\mu^{(q)}_{(u,r)}U^{I_\alpha}g\right\rangle 
=\left\langle f, \left( U^{I_\alpha}\right)^{-1} \mu^{(q)}_{(u,r)}U^{I_\alpha}g\right\rangle,
\nonumber
\end{align}
with $r=e^{-\alpha t}$. We conclude (i) I
$$
U^{I_\alpha}\mu^{I_\alpha}_{(u,t)}=\mu^{(q)}_{\left( u,e^{-\alpha t}\right)}U^{I_\alpha}.
$$

\noindent
II.  In order to prove the intertwining property of II we substitute $e^{-t}$ by $s$, and then we introduce coordinates (\ref{II_coordinates}),
\begin{align}
\left\langle f,\mu^{II}_{(u,t)}g \right\rangle &=\int_{\R_+ \times \R}f(x_1,x_2)e^{-t}\overline{g}\left(e^{-t}(x_1,x_2-tx_1)\right) 
e^{-\pi i ux_1^2}dx_1\,dx_2 \nonumber\\
&=\int_{\R_+ \times \R}f(x_1,x_2)s\overline{g}\left(sx_1,s(x_2+x_1 \log s )\right) e^{-\pi i ux_1^2}dx_1\,dx_2 \nonumber \\
&=\int_{\R_+ \times \R}f_c(y_1,y_2)s\overline{g_c}(sy_1,y_2) y_1e^{-\pi i uy_1^2}dy_1\,dy_2 \nonumber \\
&=\int_{\R_+ \times \R}y_1^\frac{1}{2}f_c(y_1,y_2)s^{\frac 12}(sy_1)^\frac{1}{2}\overline{g_c}(sy_1,y_2) e^{-\pi i uy_1^2}dy_1\,dy_2 \nonumber \\
&=\left\langle U^{II}f,\mu^{(q)}_{(u,s)}U^{II}g\right\rangle =\left\langle f, \left( U^{II}\right) ^{-1}\mu^{(q)}_{(u,s)}U^{II}g\right\rangle, \nonumber
\end{align}
with $s=e^{- t}$. Therefore (i) II follows
$$
U^{II}\mu^{II}_{(u,t)}=\mu^{(q)}_{\left( u,e^{- t}\right)}U^{II}.
$$

\noindent
III.  In order to prove the intertwining property of III we substitute $e^{-t}$ by $s$, next we introduce standard polar coordinates, and then we introduce coordinates (\ref{III_alpha_coordinates}),
\begin{align}
\left\langle f,\mu^{III_\alpha}_{(u,t)}g \right\rangle &=\int_{\R^2} f(x_1,x_2)e^{-t}\overline{g}\left(e^{-t}R_{-\alpha t}(x_1,x_2)\right) e^{-\pi i u\left(x_1^2 + x_2^2\right)}dx_1\,dx_2 \nonumber\\
&=\int_{\R^2} f(x_1,x_2)s\overline{g}\left(sR_{\alpha \log s}(x_1,x_2)\right) e^{-\pi i u \left(x_1^2 + x_2^2\right)}dx_1\,dx_2 \nonumber \\
&=\int_{\R_+\times \T}f_p(r,\theta)s\overline{g_p}(sr,\theta + \alpha \log s) r e^{-\pi i u r^2}dr\,d\theta \nonumber \\
&=\int_{\R_+\times \T}\left(r'\right)^\frac{1}{2}f_c(r',\theta')s^{\frac 12}\left(sr'\right)^\frac{1}{2}\overline{g_c}\left(sr',\theta'\right) e^{-\pi i u(r')^2}dr'\,d\theta' \nonumber \\
&=\left\langle U^{III_\alpha}f,\mu^{(q)}_{(u,s)}U^{III_\alpha}g\right\rangle
=\left\langle f,\left( U^{III_\alpha}\right)^{-1} \mu^{(q)}_{(u,s)}U^{III_\alpha}g\right\rangle, \nonumber
\end{align}
with $s=e^{- t}$. We obtain (i) III
$$
U^{III_\alpha}\mu^{III_\alpha}_{(u,t)}=\mu^{(q)}_{\left( u,e^{-t}\right)}U^{III_\alpha}.
$$ 

\noindent
IV.  In order to prove the intertwining property of IV we substitute $e^{-t}$ by $s$, next we introduce hyperbolic polar coordinates, and then we introduce coordinates (\ref{IV_alpha_coordinates}),
\begin{align}
\left\langle f,\mu^{IV_\alpha}_{(u,t)}g \right\rangle &=\int_{\R^2} f(x_1,x_2)e^{-t}\overline{g}\left(e^{-t}A_{-\alpha t}(x_1,x_2)\right) e^{-\pi i u\left(x_1^2 - x_2^2\right)}dx_1\,dx_2 \nonumber\\
&=\int_{\R^2} f(x_1,x_2)s\overline{g}\left(sA_{\alpha \log s}(x_1,x_2)\right) e^{-\pi i u \left(x_1^2 - x_2^2\right)}dx_1\,dx_2 \nonumber \\
&=\int_{\R_+\times \R}f_h(r,\theta)s\overline{g_h}(sr,\theta + \alpha \log s) r e^{-\pi i u r^2}dr\,d\theta \nonumber \\
&=\int_{\R_+\times \R}\left(r'\right)^\frac{1}{2}f_c(r',\theta')s^{\frac 12}\left(sr'\right)^\frac{1}{2}\overline{g_c}\left(sr',\theta'\right) e^{-\pi i u(r')^2}dr'\,d\theta' \nonumber \\
&=\left\langle U^{IV_\alpha}f,\mu^{(q)}_{(u,s)}U^{IV_\alpha}g \right\rangle
=\left\langle f, \left( U^{IV_\alpha}\right) ^{-1}\mu^{(q)}_{(u,s)}U^{IV_\alpha}g \right\rangle , \nonumber
\end{align}
with $s=e^{- t}$. Therefore we get (i) IV
$$
U^{IV_\alpha}\mu^{IV_\alpha}_{(u,t)}=\mu^{(q)}_{\left( u,e^{-t}\right)}U^{IV_\alpha}.
$$

\noindent
We apply (i) in order to prove (ii). In case I substitution of $e^{-\alpha t}$ by $s$ gives 
\begin{align}
-\alpha \int_{\R^2} \left| \left\langle f,\mu^{I_\alpha}_{(u,t)}g \right\rangle\right|^2 du\,e^{2\alpha t}dt 
&=-\alpha \int_{\R^2} \left|\left\langle U^{I_\alpha}f,\mu^{(q)}_{(u,e^{-\alpha t})}U^{I_\alpha}g\right\rangle \right|^2 du\,e^{2\alpha t}dt \nonumber\\
&=\int_{\R^2_+} \left|\left\langle U^{I_\alpha}f,\mu^{(q)}_{(u,s)}U^{I_\alpha}g\right\rangle \right|^2 \frac {du\,ds}{s^3}.
\nonumber
\end{align}
In cases II, III, IV we substitute $e^{-t}$ by $s$. For ${\mathcal J}=II, III_\alpha, \,IV_\alpha$ we have
\begin{align}
\int_{\R^2} \left| \left\langle f,\mu^{\mathcal J}_{(u,t)}g \right\rangle\right|^2 du\,e^{2t}dt 
&=\int_{\R^2} \left|\left\langle U^{\mathcal J}f,\mu^{(q)}_{(u,e^{-t})}U^{\mathcal J}g\right\rangle \right|^2 du\,e^{2t}dt \nonumber\\
&=\int_{\R^2_+} \left|\left\langle U^{\mathcal J}f,\mu^{(q)}_{(u,s)}U^{\mathcal J}g\right\rangle \right|^2 \frac {du\,ds}{s^3}.
\nonumber
\end{align}
Polarization formula and the fact that operators $U^{\mathcal J}$, ${\mathcal J}=I_\alpha,II,III_\alpha,IV_\alpha$, are unitary finish the proof in all cases.  The proof of (iii) follows in the same way as (ii), with integrals substituted by sums. 
\end{proof}

\begin{rem} There is a clear intuitive explanation of the choices of coordinate systems (\ref{I_alpha_coordinates}), (\ref{II_coordinates}), (\ref{III_alpha_coordinates}), (\ref{IV_alpha_coordinates}). The general guideline for the choices is: remove the effect of dilations from the second coordinate. We present relevant calculations for all of the cases I-IV. Left hand sides refer to the values of arguments in the proof of (i) of Theorem \ref{I-IV_equivalence} occurring right before the main substitution.  

\noindent
I. 
$$\left( s^{-\alpha}x_1\right)^{-\frac{\alpha +1}{\alpha}}s^{-(\alpha +1)}x_2=x_1^{-\frac{\alpha +1}{\alpha}}x_2,$$

\noindent
II. 
$$\frac{sx_2 + sx_1\log s -sx_1\log(sx_1)}{sx_1}=\frac{sx_2 + sx_1\log x_1}{sx_1}=\frac{x_2 + x_1\log x_1}{x_1},$$

\noindent
III.
$$\theta + \alpha\log s - \alpha \log (sr)= \theta - \alpha\log r,$$

\noindent
IV.
$$\theta + \alpha\log s - \alpha \log (sr)= \theta - \alpha\log r.$$
\end{rem}

\begin{cor}\label{Reproducing_Formula_iff_1d_Isometry_all_cases}
 In each of the cases ${\mathcal J}=I_\alpha,II,III_\alpha,IV_\alpha$, the system 
$\left\{\mu^{\mathcal J}_{(u,t)}\psi \right \}_{u,t\in \R}$,
with the Hilbert space and  the parameter measure being the left Haar measure, both described in (\ref{mu_I_alpha}), (\ref{mu_II}), (\ref{mu_III_alpha}), (\ref{mu_IV_alpha}), respectively, is reproducing if and only if the corresponding integral operator of the following table preserves inner products. 
\begin{equation*}
\begin{matrix}
\text{Subgroup Type}  & \text{Integral Kernel}  &  \text{Domain}  &  \text{Codomain}  \\ 
\text{ } & \text{ } & \text{ } & \text{ } \\
\text{I},\alpha \in [-1,0) 
& 
\psi\left(r, r^{\frac{\alpha +1}{\alpha}}y\right)
&
L^2(\R, dy)
& L^2\left(\R_+, r^{\frac{1-\alpha}{\alpha}}dr\right) \\
\text{ } & \text{ } & \text{ } & \text{ } \\
\text{II} 
&
\psi\left(r, ry+r\log r\right)
& L^2(\R, dy)
&  L^2\left(\R_+, \frac{dr}{r}\right) \\
\text{ } & \text{ } & \text{ } & \text{ } \\
\text{III},\alpha \in [0,\infty) 
&
\psi_p\left(r, y+\alpha \log r\right)
&  L^2(\T, dy)
&  L^2\left(\R_+, \frac{dr}{r}\right) \\
\text{ } & \text{ } & \text{ } & \text{ } \\
\text{IV},\alpha \in [0,\infty) 
&
\psi_h\left(r, y+\alpha \log r\right)
&    L^2(\R, dy)
&  L^2\left(\R_+, \frac{dr}{r}\right)
\end{matrix}
\end{equation*}
By $\psi_p$, $\psi_h$ we denote the representations of $\psi$ in polar and hyperbolic polar coordinates, respectively.
\end{cor}

\begin{proof} The result is a direct consequence of Theorem \ref{I-IV_equivalence} (ii) and Corollary \ref{Reproducing_Formula_iff_1d_Isometry_q_case}.
\end{proof}

\begin{cor}\label{Reproducing_Discrete_System_iff_1d_Isometry_all_cases}
 In each of the cases ${\mathcal J}=I_\alpha,II,III_\alpha,IV_\alpha$, the system 
$\left\{\mu^{\mathcal J}_{\lambda}\psi \right \}_{\lambda\in \Lambda^{\mathcal J}}$,
with the Hilbert space described in (\ref{mu_I_alpha}), (\ref{mu_II}), (\ref{mu_III_alpha}), (\ref{mu_IV_alpha}), respectively,  is a Parseval frame,  if and only if the corresponding integral operator of the following table preserves inner products for almost every $r\in \R_+$. Domains of the operators are the same as in Corollary \ref{Reproducing_Formula_iff_1d_Isometry_all_cases}. Codomains are weighted $l^2(\Z)$ with the indicated weight.
\begin{equation*}
\begin{matrix}
\text{Subgroup Type}  & \text{Integral Kernel}  &  \text{Weight}  \\ 
\text{ } & \text{ } & \text{ } \\
\text{I},\alpha \in [-1,0) 
& 
2^{-1/2}r^{\frac{1}{2\alpha}}\psi\left(2^{k/2}r, \left(2^{k/2}r\right)^{\frac{\alpha +1}{\alpha}}y\right)
& 2^{\frac{k}{2\alpha}} \\
\text{ } & \text{ } & \text{ } & \text{ } \\
\text{II} 
&
2^{-1/2}\psi\left(2^{k/2}r, 2^{k/2}ry+2^{k/2}r \log \left(2^{k/2}r\right)\right)
&  1 \\
\text{ } & \text{ } & \text{ } \\
\text{III},\alpha \in [0,\infty) 
&
2^{-1/2}\psi_p\left(2^{k/2}r, y+\alpha \log \left(2^{k/2}r\right) \right)
&  1 \\
\text{ } & \text{ } & \text{ } \\
\text{IV},\alpha \in [0,\infty) 
&
2^{-1/2}\psi_h\left(2^{k/2}r, y+\alpha \log \left(2^{k/2}r\right) \right)
&  1
\end{matrix}
\end{equation*}
By $\psi_p$, $\psi_h$ we denote the representations of $\psi$ in polar and hyperbolic polar coordinates.
We assume that for almost every $y \in Y$, where $Y=\R$ for ${\mathcal J}=I_\alpha,II,IV_\alpha$, and $Y=\T$ for ${\mathcal J}=III_\alpha$,  $\text{ess-supp}\, \psi(\cdot,y)\subset [0,1]$ in cases  ${\mathcal J}=I_\alpha,II$, and 
$\text{ess-supp}\, \psi_p(\cdot,y)\subset [0,1]$, $\text{ess-supp}\, \psi_h(\cdot,y)\subset [0,1]$ in cases  ${\mathcal J}=III_\alpha,IV_\alpha$, respectively.
\end{cor}

\begin{proof} The result is a direct consequence of Theorem \ref{I-IV_equivalence} (iii) and Corollary \ref{Reproducing_Discrete_System_iff_1d_Isometry_q_case}.
\end{proof}

We define $\psi^{{\mathcal J}, D_\R}=\left( U^{\mathcal J}\right)^{-1}U\psi^{D_\R}$, for $\,{\mathcal J}=I_\alpha,II,IV_\alpha$, and $\psi^{{\mathcal J}, D_\T}=\left( U^{\mathcal J}\right)^{-1}U\psi^{D_\T}$, for $\,{\mathcal J}=III_\alpha$, where $\psi^{D_\R}$, $\psi^{D_\T}$ are defined in (\ref{2d_generating_function_R}), (\ref{2d_generating_function_T}), respectively, and $U$ is the unitary map of Proposition \ref{l_q_equivalence}. 

\begin{cor}\label{2d_Reproducing_System_D_R_D_T_J}
 Systems $\left\{ c_f^{-1}\mu^{\mathcal J}_{(u,t)}\psi^{{\mathcal J}, D_\R}\right \}_{u,t\in \R}$, 
$\,{\mathcal J}=I_\alpha,II,IV_\alpha$, $\left\{ c_f^{-1}\mu^{\mathcal J}_{(u,t)}\psi^{{\mathcal J}, D_\T} \right \}_{u,t\in \R}$, $\,{\mathcal J}=III_\alpha$, where $c_f=\left( \log 2\right)^{1/2}$,
with the Hilbert spaces and  the parameter measures the same as in Corollary \ref{Reproducing_Formula_iff_1d_Isometry_all_cases}, are reproducing.
\end{cor}
\begin{proof} The result is a direct consequence of  Theorem \ref{I-IV_equivalence} (ii) and Corollary \ref{2d_Reproducing_Systems_D_R_D_T_q}.
\end{proof}

\begin{cor}\label{2d_Orthonormal_Basis_D_R_D_T_J}
 Systems $\left\{\mu^{\mathcal J}_{\lambda}\psi^{{\mathcal J}, D_\R}\right \}_{\lambda\in \Lambda^{\mathcal J}}$, 
$\,{\mathcal J}=I_\alpha,II,IV_\alpha$, $\left\{\mu^{\mathcal J}_{\lambda}\psi^{{\mathcal J}, D_\T} \right \}_{\lambda\in \Lambda^{\mathcal J}}$,  \newline $\,{\mathcal J}=III_\alpha$, 
with the Hilbert spaces the same as in Corollary \ref{Reproducing_Discrete_System_iff_1d_Isometry_all_cases}, are orthonormal bases.
\end{cor}
\begin{proof}  The result is a direct consequence of  Theorem \ref{I-IV_equivalence} (iii) and Corollary \ref{2d_Orthonormal_Bases_D_R_D_T_q}
\end{proof}

\begin{center}
\text{A\scriptsize{CKNOWLEDGEMENTS}}
\end{center}

The authors would like to thank Hartmut F\"uhr for many pertinent comments on an early version of the current paper.

Margit Pap  was supported by the European Union, co-financed by the European Social Fund. EFOP-3.6.1.-16-2016-00004.


\begin{thebibliography}{Comment}

\bibitem[1]{AAG}S.T. Ali, J.-P. Antoine, J.-P Gazeau, {\em Coherent States, Wavelets, and Their Generalizations}, Second Edition, Theoretical and Mathematical Physics, Springer, New York, 2014.

\bibitem[2]{DM&Co2}G.S. Alberti, L. Balletti, F. De Mari, E. De Vito, {\em Reproducing Subgroups of $Sp(2,\R)$, Part I: Algebraic Classification}, J. Fourier Anal. Appl., {\bf 19}(4) (2013), 651--682.

\bibitem[3]{DM&Co3}G.S. Alberti, F. De Mari, E. De Vito, L. Mantovani, {\em Reproducing Subgroups of $Sp(2,\R)$, Part II: Admissible Vectors}, Monatsh. Math., {\bf 173}(3) (2014), 261--307.

\bibitem[4]{Aro}N. Aronszajn,{\em  Theory of Reproducing Kernels}, Trans. Am. Math. Soc., {\bf 68} (3) (1950), 337--404

\bibitem[5]{CDMNT1}E. Cordero, F. De Mari, K. Nowak, and A. Tabacco, {\em Analytic Features of Reproducing Groups for the Metaplectic Representation}, J. Fourier Anal. Appl. 12 (2006), , 157–-179.

\bibitem[6]{CDMNT2}E. Cordero, F. DeMari, K. Nowak, and A. Tabacco, {\em Reproducing Subgroups for the Metaplectic Representation}, Operator Theory: Advances and Applications 164 (2006), 227--244.

\bibitem[7]{CDMNT3}E. Cordero, F. De Mari, K. Nowak, and A. Tabacco, {\em Dimensional Upper Bounds for Admissible Subgroups of the Metaplectic Representation}, Math. Nachr. 283 (2010), 982–-993.

\bibitem[8]{CoTa} E. Cordero, A. Tabacco, {\em Triangular Subgroups of $Sp(d,\R)$ and Reproducing Formulae}, J. Func. Anal., {\bf 264}(9) (2013), 2034--2058.

\bibitem[9]{Dau}I. Daubechies, {\em Ten Lectures on Wavelets},
CBMS-NSF Regional Conference Series, no. 6, SIAM, Philadelphia, 1992.

\bibitem[10]{DeGo}M. De Gosson, {\em Symplectic Geometry and Quantum Mechanics}, Operator Theory, Advances and Applications, Vol. 166, Birkh\"auser, Basel, 2006. 

\bibitem[11]{DM&Co1}F. De Mari and E. De Vito, {\em Admissible Vectors for Mock Metaplectic Representations}, Appl. Comput. Harmon. Anal. 34 (2013), 163–-200.

\bibitem[12]{DM&Co4} F. De Mari, E. De Vito, {\em The Use of Representations in Applied Harmonic Analysis}, S. Dahlke et al. (eds.), Harmonic and Applied Analysis, Applied and Numerical Harmonic Analysis 68, Springer, New York, 2015.

\bibitem[13]{DMNo1}F. De Mari, K. Nowak, {\em Analysis of the Affine 
Transformations of the Time-Frequency Plane}, Bull. Austral. Math. Soc., {\bf 63}(2) (2001), 195--218.

\bibitem[14]{DMNo2}F. De Mari, K. Nowak, {\em Canonical Subgroups of ${\mathbb H}_1\rtimes SL(2,\R)$}, Bol. U.M.I. Sez. B, {\bf 5}(8) (2002), 405--430.

\bibitem[15]{Fol1}G. Folland, {\em Harmonic Analysis in Phase Space},
Princeton University Pres, Princeton, 1989.

\bibitem[16]{Fol2}G. Folland, {\em A Course in Abstract Harmonic Analysis},
 Stud. Adv. Math., CRC Press, Boca Raton, FL, 1995.

\bibitem[17]{Fuh}H. F\"uhr, {\em Abstract Harmonic Analysis of Continuous Wavelet Transforms}, Lecture Notes in Mathematics, Vol. 1863, Springer, New York, 2005. 

\bibitem[18]{Gro}K. Gr\"ochenig, {\em Foundations of Time-Frequency Analysis}, Birkh\"auser, Boston, 2001.

\bibitem[19]{Loj}S. {\L}ojasiewicz, {\em An introduction to the Theory of Real Functions}, Wiley, New York, 1988.

\bibitem[20]{NoPa}K. Nowak, M. Pap, {\em Two-dimensional Shannon type expansions via one-dimensional affine and wavelet lattice actions}, arXiv, https://arxiv.org/pdf/1611.05779.pdf, to appear in Colloquium Mathematicum.

\bibitem[21]{Rud}W. Rudin, {\em Real and Complex Analysis}, Third Edition, McGraw-Hill, New York, 1987.

\bibitem[22]{Woj}P. Wojtaszczyk, {\em A Mathematical Introduction to Wavelets},
Cambridge University Press, Cambridge, 1997. 

\end{thebibliography}
\end{document}